\documentclass[11pt,oneside]{amsart}

\usepackage{amsmath,ifthen, amsfonts, amssymb, amsopn}

\usepackage{graphicx}

\newcommand{\showcomments}{yes}

\newsavebox{\commentbox}
{\ifthenelse{\equal{\showcomments}{yes}}%
{\footnotemark
        \begin{lrbox}{\commentbox}
        \begin{minipage}[t]{1.25in}\raggedright\sffamily\tiny
        \footnotemark[\arabic{footnote}]}
{\begin{lrbox}{\commentbox}}}%
{\ifthenelse{\equal{\showcomments}{yes}}%
{\end{minipage}\end{lrbox}\marginpar{\usebox{\commentbox}}}
{\end{lrbox}}}

\newtheorem{thm}{Theorem}[section]
\newtheorem{lem}[thm]{Lemma}

\newtheorem{cor}[thm]{Corollary}

\newtheorem{prop}[thm]{Proposition}

\theoremstyle{definition}
\newtheorem{defn}[thm]{Definition}
\newtheorem{rem}[thm]{Remark}
\newtheorem{exmp}[thm]{Example}

\newtheorem{question}[thm]{Question}

\newcommand{\field}[1]{\mathbb{#1}}
\newcommand{\integers}{\ensuremath{\field{Z}}}

\newcommand{\naturals}{\ensuremath{\field{N}}}
\newcommand{\reals}{\ensuremath{\field{R}}}
\newcommand{\Euclidean}{\ensuremath{\field{E}}}

\newcommand{\size}[1]{\ensuremath{\vert #1 \vert}}

\begin{document}

\title[Isometries of $CAT(0)$ cube complexes are semi-simple]{Isometries of $CAT(0)$ cube complexes are semi-simple}

\author[F.~Haglund]{Fr\'{e}d\'{e}ric Haglund}
           \address{
Laboratoire de Math\'{e}matiques\
Universit\'{e} de Paris XI (Paris-Sud)\\
91405 Orsay\\
FRANCE\\}
           \email{frederic.haglund\@math.u-psud.fr}

\date{\today}

\begin{abstract}
We show that an automorphism of an arbitrary $CAT(0)$ cube complex either has a fixed point or preserves some combinatorial axis. It follows that when a group contains a distorted cyclic subgroup, it admits no proper action on a discrete space with walls.   
\end{abstract}

\maketitle

\tableofcontents

\section{Introduction.}

The notion of {\em a space with walls} was introduced in
\cite{HaglundPaulin98}. We intended to develop a common geometric langage
for various classical combinatorial structures, like the Davis-Moussong
complex of a Coxeter group (see \cite{Davis83}, \cite{Moussong88}), simply-connected
polygonal complexes all of whose polygons have an even number of sides or
$CAT(0)$ cube complexes. A space with walls is a set $V$ together with a
collection $\mathcal H$ of separating objects: the walls. It is required
that two points are separated by finitely many walls.

 This notion was generalized in \cite{CherixMartinValette03} where the definition of a 
{\em space with measured walls} was given. Here there is a measure on the
set of walls such that the measure of the set of walls separating two
given points is finite. If a group $G$ admits a proper action on a
space with measured walls then it has the Haagerup property
(\cite{CherixMartinValette03}[Proposition 1]). At the end of their
paper the authors of \cite{CherixMartinValette03} asked :

\begin{question}\label{quest:discrete}
\center
 Does  a discrete group acting properly on a space with measured walls
necessarily admit a proper action on a {\em discrete} space with walls ?
\end{question}

 The Baumslag-Solitar group $BS(m,n)$ with parameters $m,n\in\naturals^*$
has the following presentation:

$$BS(m,n):=\ < \ a,b\ \vert\ ba^mb^{-1}=a^n\ >$$

The groups $BS(m,n)$ all act
properly on a space with measured walls. We briefly recall such an action. First we have the action on the (locally finite) Bass-Serre tree $T_{m,n}$ of $BS(m,n)$ seen as the $HNN$-extension of the cyclic group generated by $a$ by the isomorphism $b$ sending the subgroup $\langle a^m\rangle$ onto the subgroup $\langle a^n\rangle$. This gives an action on a simplicial tree, thus a discrete space with walls. But this action is not proper since $a$ fixes a vertex. Now we can also see $a$ as a unit translation of the real line $\reals$, and $b$ as a well-choosen non trivial homothety of $\reals$. For example using matrices we may let $A=\begin{pmatrix}1&1\cr 0&1\cr \end{pmatrix}, B=\begin{pmatrix}{\rm e}^\beta&0\cr 0&{\rm e}^{-\beta}\cr \end{pmatrix}$ with $({\rm e}^\beta)^2=\frac{n}{m}$, then $a\mapsto A,b\mapsto B$ induces a representation of $BS(m,n)$ into $PSL(2,\reals)$, the isometry group of $\field H^2$. Thus $BS(m,n)$ is represented as a parabolic group of isometries of $
 \field H^2$, in such a way that the subgroup $\langle a\rangle$ acts properly. Now $\field H^2$ has a natural $PSL(2,\reals)$-invariant structure of space with measured walls (see section 3 in \cite{CherixMartinValette03}), and it is readily seen that the action of $BS(m,n)$ on the product space with measured walls $T_{m,n}\times \field H^2$ is proper  (A. Valette explained this to me in an oral communication, see also \cite{GalJanuszkiewicz03})
\bigskip

 In this paper we answer in the negative to Question~\ref{quest:discrete}
by proving the following:
\begin{thm}\label{thm:BSnoproperwall}

 For $m\neq n$ the group $BS(m,n)$ has no proper  action on a (discrete)
space with walls.

\end{thm}

 We want to show that for every action of $BS(m,n),m\neq n$ on a space
with walls $(V,{\mathcal  H})$, all orbits of the infinite cyclic group
generated by $b$ are bounded. To do so we consider the $CAT(0)$ cube
complex $X$ naturally associated to $(V,{\mathcal  H})$ and the induced
action of $BS(m,n)$ on it. This complex was introduced simultaneously in
\cite{ChatterjiNiblo04, NicaCubulating04},  thus generalizing the
previous constructions of  \cite{Sageev95} and \cite{WiseSmallCanCube04}.

 Let us review some elementary properties of the $CAT(0)$ cube complex
$X$ (see \cite{ChatterjiNiblo04, NicaCubulating04} for details). There is an embedding $V\to X^0$ which is an isometry when $V$ is
equipped with the wall distance (so $d(v,v')$ is the number of walls
separating $\{v,v'\}$) and $X^0$ is equipped with the combinatorial
distance of the 1-skeleton $X^1$. An automorphism $f$ of the space with
walls  $(V,{\mathcal  H})$ extends to an automorphism $\bar f$ of $X$.
The property of having  bounded orbits is equivalent for $f$ or for $\bar
f$. And the map $f\mapsto\bar f$ is an injective extension morphism ${\rm
Aut}(V,{\mathcal  H})\to {\rm Aut}(X)$. It follows that
Theorem~\ref{thm:BSnoproperwall} is a consequence of its analogue for
$CAT(0)$ cube complexes:

\begin{thm}\label{thm:BSnopropercube}

 For $m\neq n$ the group $BS(m,n)$ has no proper  action on a $CAT(0)$
cube complex.

\end{thm}
 Theorem~\ref{thm:BSnopropercube} is easy to prove when the cube complex
is finite dimensional. In order to handle the general case we establish a
classification of automorphisms of arbitrary $CAT(0)$ cube complexes. Our
main result is:

 \begin{thm} \label{thm:classif}
Every automorphism of a $CAT(0)$ cube complex acting stably without
inversion is either  combinatorially elliptic or  combinatorially
hyperbolic. 

\end{thm}
 This means that either the automorphism $f$ fixes a vertex, or $f$
preserves a combinatorial geodesic on which it has a positive translation
length $\delta$, and for any  vertex $v$ in the cube complex we have $d(v,f(v))\ge \delta$.
The condition of acting stably without inversion is always fullfilled in the
cubical subdivision.

In fact, as was suggested to us by C. Dru\c tu, using the fact that $a^{m^k}=b{a^{n^k}}b^{-1}$ in $BS(m,n)$ we see that the subgroup $\langle a\rangle$ is distorted, so we can deduce Theorem~\ref{thm:BSnopropercube} from the following:

\begin{thm}\label{thm:distortednopropercube}
Let $\Gamma$ denote a finitely generated group containing an element $a$ such that the subgroup $\langle a\rangle$ is distorted in $\Gamma$. Then the infinite subgroup $\langle a\rangle$ has a fixed point in every action of
$\Gamma$  on a $CAT(0)$
cube complex. Consequently $\Gamma$ has no proper action on a discrete-space with walls.
\end{thm}

For example we deduce:

\begin{cor}\label{cor:HBGnopropercube}
Let $H=\langle a,b,c\ \vert\  [a,c]=[a,b]=1,[b,c]=a\rangle$ denote the (discrete) Heisenberg group and let $H\to G$ denote a morphism which is injective on the distorted subgroup $\langle a\rangle$. Then $G$ has no proper  action on a discrete-space with walls.
\end{cor}

Note that since the  Heisenberg group $H$ is nilpotent it is amenable and thus acts properly on a space with measured walls (see \cite{CherixMartinValette03}, Theorem 1 (5)). So $H$ itself is an other negative answer to Question~\ref{quest:discrete}. 

Theorem~\ref{thm:classif} says that when $\integers$ acts freely on a $CAT(0)$ cube complex, then there are invariant axes for $\integers$,  each axis being a geodesic isomorphic to the subdivided line. We conclude with the following natural 

\begin{question}
Let $\integers^n$ act freely on a $CAT(0)$ cube complex. Does there exist a  $\integers^n$-invariant combinatorially geodesic subcomplex isomorphic to the standard cubulation of $\Euclidean^n$ ?
\end{question}

\bigskip

In Section~\ref{sec:geometry cube complex}
 we review some fairly classical facts about the geometry of
$CAT(0)$ cube complex, first considered as combinatorial objects by Gromov in \cite{Gromov87}.
Thus people familiar with $CAT(0)$ cube complexes may skip it.
We concentrate on the  combinatorial distance
between vertices and insist on the role of hyperplanes. All the results of this section are classical, except that no assumption of finite dimension is made.

 In Section~\ref{sec: comb trans length} we define the models (elliptic,
hyperbolic) for all inversion-free automorphisms of a $CAT(0)$ cube
complex.

 In Section~\ref{sec:inversion} we generalize to arbitrary $CAT(0)$ cube
complexes the notion of action without inversion that occurs so often in
the study of groups acting on trees (see \cite{Serre80}). In particular
we show that a group acting on a $CAT(0)$ cube complex has no inversion
on the cubical subdivision  (see Lemma~\ref{lem:subdinversion}). An automorphism acts stably without inversion if every power of the automorphism acts without inversion.

 In Section~\ref{sec:axisautom} we prove that if an automorphism of a
$CAT(0)$ cube complex  preserves a combinatorial geodesic then its
minimal (combinatorial) displacement in the complex is the same as on the
geodesic.

 In Section~\ref{sec:classification} we prove Theoem~\ref{thm:classif}.
So we show that an automorphism acting stably without inversion and without fixed point
preserves a combinatorial geodesic.

 And in Section~\ref{sec:application} we deduce
Theorem~\ref{thm:distortednopropercube}.

\bigskip

I would like to thank Estelle Souche, Yves Stalder, Alain Valette, Cornelia Dru\c tu and Dani Wise for helpfull discussions on this subject.

\section{Geometry of $CAT(0)$ cube complexes.}\label{sec:geometry cube
complex}

\subsection{Cube complexes and non-positive curvature conditions}

The following notion of {\em cube complexes} is equivalent to the notion
of {\em cubical complexes} introduced in
\cite{BridsonHaefliger}[Definition 7.32 p 112].

\begin{defn}[cube complexes] Let $X$ denote some set.
{\em A parametrized cube of $X$} is an embedding $f:C\to X$, where $C$ is
some 
euclidean cube (a cube of a euclidean space).
{\em A face} of a parametrized cube $f:C\to X$ is a
restriction of $f$ to one of the faces of $C$. Two parametrized cubes
$f:C\to X,f':C'\to X$ are {\em isometric} whenever there is an isometry
$\varphi :C\to C'$ such that $f\varphi=f'$.

{\em A  cube complex} is a set $X$ together with a family $(f_i)_{i\in I}$
of  parametrized cubes $f_i:C_i\to X$  such that

\begin{enumerate}
\item(covering) each point of $X$ belongs to the range of some $f_i$
\item(compatibility) for any two maps $f_i,f_j$ either  $f_i(C_i)\cap
f_j(C_j)=\emptyset$ or $f_i,f_j$ have isometric faces
$f_{ij}:C_{ij}\subset C_i\to X,f_{ji}:C_{ji}\subset C_j\to
X$ such that $f_i(C_{ij})=f_j(C_{ji})=f_i(C_i)\cap f_j(C_j)$
\end{enumerate}

The {\em cube} associated to the parametrized cube $f_i:C_i\to X$ is the
image $f_i(C_i)$. Note that if two parametrized cubes $f_i:C_i\to
X,f_j:C_j\to X$ have the same image, then by the compatibility condition
above the cubes
$C_i,C_j$ are isometric, in particular they have the same dimension. This
dimension, say $k$,  will be called the {\em dimension} of the cube
$f_i(C_i)=f_j(C_j)$: we will say for short that the cube is {\em a $k$-cube}.
The 0-cubes are the {\em vertices}, the 1-cubes are the {\em edges}, the
2-cubes are  the {\em squares} ...

{\em The interior of a cube $f_i(C_i)$ of
$X$} is the image under $f_i$ of the interior of $C_i$ (independant of
the parametrized cube $f_j:C_j\to X$ such that $f_j(C_j)=f_i(C_i)$ again
by the compatibility condition). Note that by the covering condition each
point $p$ of $X$ is contained in the interior of some cube, and by the
compatibility this cube is unique (we will denote it by $C(p)$).

 When only one family $(f_i)_{i\in I}$ is considered on $X$ we will say
by abuse of langage that {\em $X$ is a cube complex}, and we will denote
by $X^k$ the union of all $k$-cubes of $X$. This is the {\em $k$-skeleton
of
$X$}. More generally {\em a subcomplex of $X$} is a union of cubes of
$X$. Note that each subcomplex inherits a natural structure
of cube complex.

 For any two cube complexes $X,Y$, a map $f:X\to Y$ is said to be {\em
combinatorial} whenever for each parametrized cube $f_i:C_i\to X$, the
composite $ff_i$ is isometric to a
parametrized cube of $Y$. In particular such a map sends vertices to
vertices, edges to edges ...
Observe that the natural inclusions of subcomplexes are combinatorial.
Note also that a combinatorial map $f:X\to Y$ sends the interior of a
cube $C\subset X$ bijectively onto the interior of $f(C)$. 

{\em The automorphism group of a cube complex $X$} is the set of
combinatorial bijections $X\to X$ (this is indeed a subgroup of the
permutation group of $X$).

\end{defn}

\begin{exmp}
Assume that $\mathcal H$ is a Hilbert space and that $\mathcal C$ is a
collection of unit euclidean cubes of $\mathcal H$, such that for two
cubes of  $\mathcal C$, their intersection is either empty or a cube of
$\mathcal C$. Then the union $X$ of all cubes $C\in{\mathcal C}$ has a
natural structure of cube complex where the parametrized cubes are the
inclusion $C\to X,C\in{\mathcal C}$.

\end{exmp}

\begin{defn}[subdivisions]
Let $C$ denote some euclidean cube. For short we will call {\em  barycenter of $C$} the barycenter of the set of vertices of $C$ (with equal unit weights).
{\em A vertex of the barycentric
subdivision of
$C$} is the  barycenter $b_F$ of some face $F$ of $C$. {\em A simplex of the
barycentric subdivision of
$C$} is the simplex affinely
generated by vertices $b_{F_0},\dots,b_{F_k}$, where the faces $F_i$
satisfy $F_0\subset \dots\subset {F_k}$. This defines a  simplicial complex ${C'}_{\rm simpl}$  called {\em the simplicial barycentric subdivision of
$C$}, whose geometric realization is
identical with $C$. Note that any isometry
$C\to D$ of unit euclidean cubes induces an isomorphism ${C'}_{\rm simpl}\to {D'}_{\rm simpl}$.

Let $X$ denote a cube complex. The {\em simplicial barycentric subdivision of $X$} is 
the simplicial complex ${X'}_{\rm simpl}$ whose cells are restrictions of the $f_i:C_i\to X$ to the simplices of ${{C_i}'}_{\rm simpl}$.

Given  two comparable faces $F\subset G\subset C$ of a eulidean cube, the union of simplices $\{b_{F_0},\dots,b_{F_k}\}$ of $C'$ satisfying $F\subset F_0,F_k\subset G$ is in fact a eulidean cube, which we call  {\em a cube of the barycentric subdivision of
$C$}.
This decomposition of the cube $C$ into smaller cubes endows $C$ with a structure of cube complex, which we call   {\em the cubical subdivision of
$C$}, denoted by ${C'}$. There is still an isomorphism ${C'}\to {D'}$  induced by any isometry
$C\to D$.

Let $X$ denote a cube complex. The {\em cubical subdivision of $X$} is 
the cube complex ${X'}$ whose parametrized cubes are restrictions of the $f_i:C_i\to X$ to the cubes of the cubical subdivision of
${C_i}$.

\end{defn}

\begin{defn}[links of vertices]

Let $X$ denote a cube complex, and let $v$ denote some vertex.
 The collection of cubes containing $v$ (distinct of $\{v\}$) is an
abstract simplicial complex, whose set of vertices is the set of edges of
$X$ containing $v$. We will  denote this simplicial complex by ${\rm
link}(v,X)$ (the {\em link of $v$ in $X$}).

 The cube complex $X$ is {\em combinatorially non positively curved}  if 
each vertex link is
{\em flag} (that is each complete subgraph is the
1-skeleton of a simplex). We say that $X$ is {combinatorially $CAT(0)$} whenever $X$ is 
combinatorially non positively curved and simply-connected. Following Sageev we will rather say {\em cubing} instead of combinatorially $CAT(0)$ cube complex

\end{defn}

\begin{defn}[the pseudometric of a cube complex, see
\cite{BridsonHaefliger} 7.38 p 114]

Let $X$ denote a cube complex. {\em A piecewise geodesic of $X$ with
endpoints $x,y$} is a map
$c:[a,b]\to X$ such that $c(a)=x,c(b)=y$, there is a subdivision
$a=t_0\le\dots\le t_n=b$, a sequence of parametrized cubes $f_1:C_1\to
X,\dots,f_n:C_n\to X$ and a sequence of isometries $(c_i:[t_{i-1},t_i]\to
C_i)_{1\le i\le n}$ satisfying $f_ic_i=c$ on $[t_{i-1},t_i]$ for $1\le
i\le n$.

The {\em length} of the piecewise geodesic is $\size{b-a}$.

We define a pseudometric $d$ on $X$ by setting $d(x,y)=$ the infimum of
the lengths of piecewise geodesics of $X$ with endpoints $x,y$.

\end{defn}

Here are now two results linking the metric and the combinatorial viewpoint on cube complexes.

\begin{lem}\label{lem:geodesiclengthmetric}
The pseudometric $d$ on a combinatorially non positively curved and simply-connected cube complex  $X$ is a geodesic length metric. 
\end{lem}

\begin{lem}[\cite{Gromov87}]\label{lem:equivCAT0}
 Let $X$ denote some  cube complex. Then $X$  is a combinatorially non
positively curved if and only if the length metric $d$ on $X$  is locally
$CAT(0)$. In particular a cube complex is metrically $CAT(0)$ if and only if it is
combinatorially $CAT(0)$.
\end{lem}

The two previous lemmas have been established for cube complexes whose cubes have all dimension $\le n$ ; see for example \cite{BridsonHaefliger}. This result is so classical for people working on cube complexes that they usually identify the metric condition $CAT(0)$ with its combinatorial analogue. But in fact the equivalence between the metric and the combinatorial condition has never been checked explicitly for general cube complexes. People just take the combinatorial non positively curvature condition as a {\em definition} of "locally $CAT(0)$". 
For instance, when Chatterji-Niblo and Nica achieve the geometrization of spaces with walls (\cite{ChatterjiNiblo04, NicaCubulating04}), turning these to $CAT(0)$ cube complexes, they in fact check the combinatorial condition, and do not tell us anything about the metric - although Gromov's hypothesis of finite dimensionality usually fails. Gerasimov is very near to proving the Lemmas in full generality, but he does not do it explicitly (see \cite{Gerasimov97}). In a private communication Michah Sageev has told me that Y. Algom, a student of his, had obtained a proof (unpublished).

We do not insist since in the sequel we will be only concerned with the {\em combinatorial} geometry of $CAT(0)$ cube complexes (or cubings), which we recall in the next section. Our slogan is: in $CAT(0)$ cube complexes the combinatorial geometry is as nice as the $CAT(0)$ geometry.  In fact the result of this paper is that, in some context, it is even nicer.

{\em  In the sequel we do not make any restriction on the cubing $X$: in  particular we do not assume $\dim X<\infty$.}

\subsection{Hyperplanes, combinatorial distance and convex subcomplexes}

\begin{defn}\label{defn:elemhom}
A {\em combinatorial path} of a cube complex $X$ is a sequence $\gamma=(v_0,v_1,\dots,v_n)$ of vertices of $X$ such that for each $i=0,\dots,n-1$ either $v_{i+1}=v_i$ or  $v_{i+1},v_i$ are the two (distinct) endpoints of some edge of $X$. The {\em initial point of $\gamma$} is $v_0$, {\em the terminal point of $\gamma$} is $v_n$ and {\em the length of $\gamma$} is $n$. If for each $i=0,1,\dots,n-1$ we have $v_{i+1}\neq v_i$ we say that $\gamma$ is {\em non stuttering}.

 When $\gamma=(v_0,v_1,\dots,v_n),\gamma'=(w_0,w_1,\dots,w_m)$ are two combinatorial paths such that the terminal point of $\gamma$ is the initial point of $\gamma'$ we define as usual the product $\gamma.\gamma'$  to be the path $(v_0,v_1,\dots,v_{n-1},v_n=w_0,w_1,\dots,w_m)$.

The {\em combinatorial distance between two vertices $x,y$} of a connected cube complex is the minimal length of a combinatorial path joining $x$ to $y$. It will be denoted by $d(x,y)$, and a path of length $d(x,y)$ will be called {\em a (combinatorial) geodesic}. We note that $d(x,y)$ is also the minimal length of a non stuttering combinatorial path joining $x$ to $y$ (in other words geodesics are non stuttering).

Sequences $(p_n)_{n\in\integers}$ of vertices of the cube complex such that $d(p_n,p_m)=\size{m-n}$ will also be called {\em (infinite) geodesics}.

\end{defn}

\begin{defn}[convex subcomplexes]

Let $X,Y$ denote  cube complexes and let $f:X\to Y$ denote a combinatorial map. We say that $f$ is a {\em local isometry} if for each vertex $v$ of $X$ the induced simplicial map $f_v:{\rm link}(v,X)\to  {\rm link}(f(v),Y)$ is injective and has full image (recall that a subcomplex $L\subset K$ of a simplicial complex is {\em full} whenever each simplex of $K$ whose vertices are in $L$ in fact belongs to $L$).

We say that a subcomplex of a cube complex is {\em locally convex} if the inclusion map is a local isometry.

We say that a subcomplex $Y$ of a cube complex $X$ is {\em  (combinatorially) convex} if it is connected, and any (combinatorial) geodesic between two vertices of $Y$ has all of its vertices inside $Y$.

\end{defn}
\begin{rem}
\begin{enumerate}
\item A fundamental property of cubings is that they have plenty of non-trivial convex subcomplexes.
\item In fact a connected subcomplex $Y\subset X$ of  a cubing  is locally convex in the sense of the above definition iff it is geodesically convex for the $CAT(0)$ metric. Thus there is no confusion with the term convex.
\end{enumerate}

\end{rem}

We now introduce the most important tool for studying the combinatorial geometry of cubings:

\begin{defn}[walls, hyperplanes]
Let $X$ denote any cube complex. Two edges $a,b$ of $X$ are said to be {\em elementary parallel} whenever they are disjoint but contained in some (necessarily unique) square of $X$. We call {\em parallelism} the equivalence relation on the set of edges of $X$ which is generated by elementary parallelisms. A {\em wall of $X$} is an equivalence relation for the parallelism relation. When an edge $e$ belongs to some wall $W$ we say that {\em $W$ passes through $e$}, or that {\em $W$ is dual to $e$}. We also say that {\em a cube $C$ of $X$ is dual to the wall $W$} when $C$ contains an edge $e$ to which $W$ is dual.

Let $C$ denote some euclidean cube of dimension $n$, and let $\Euclidean$ denote the ambient euclidean space. For every edge $e$ of $C$ with endpoints $p,q$ we consider the hyperplane of $\Euclidean$ consisting in points which are at the same distance of $p$ and $q$. Then the intersection of this hyperplane with $C$ is a euclidean cube, whose cubical subdivision is a subcomplex of $C'$. We denote this subcomplex of $C'$ by $h_e$, and call it {\em the hyperplane of $C$ dual to $e$}. Observe that $h_e=h_{e'}$ iff $e$ and $e'$ are parallel. Note also that the hyperplane of a segment consist in its midpoint.

Now let $e$ denote some edge of  a cube complex $X$, and let $W$ denote the wall through $e$.
For each parametrized cube $f:C\to X$ and each edge $a$ of $C$ such that $e\parallel f(a)$, we consider the image under the induced combinatorial map $f:C'\to X'$ of the hyperplane $h_a$ of $C$ dual to $a$.  The union of all these $f({h_a})$ is called {\em a hyperplane of $X$}, it  will be denoted by $H_e$, and we will say that $H_e$  is {\em dual to $e$}. (Note that a hyperplane of $X$ is a subcomplex of the subdivision $X'$.) Clearly $H_e=H_{e'}$ iff $e$ and $e'$ are parallel in $X$. In other words the set of edges to which a given hyperplane is dual consists in a wall. Thus walls and hyperplanes are in one-to-one correspondance. We will say that a cube $C$ of $X$ is {\em dual} to some hyperplane $H$ when $C$ contains an edge $e$ to which $H$ is dual.

Let $H$ denote some hyperplane of a cube complex $X$. The {\em neighbourhood of $H$} is the union of all cubes dual to $H$. We will denote it by $N_H$.

\end{defn}

Combining results of Sageev we get the following description of hyperplane neighbourhoods:

\begin{thm}\label{thm:hyperplanes}[see \cite{Sageev95}, ]
Let $X$ be a cubing and let $H$ denote some hyperplane of $X$ with neighbourhood  $N_H$.
\begin{enumerate}
\item $H$ separates $X$ into two connected components
\item\label{convneighbourhood} $N_H$ is (combinatorially) convex in $X$
\item\label{symmetry} $N_H$ admits an automorphism $\sigma_H$ that fixes pointwise $H$ and exchanges the endpoint of each edge dual to $H$
\end{enumerate}
\end{thm}

\begin{proof}
\begin{enumerate}
\item This is Theorem 4.10 in \cite{Sageev95}.
\item By the separation property above, it follows that the union of all cubes of $N_H$ disjoint of $H$ consists in the disjoint union of two connected subcomplexes, which we call the boundary components of $N_H$.  Then  Theorem 4.13 of \cite{Sageev95} tells us that each boundary component of $N_H$ is a combinatorially convex subcomplex. The combinatorial convexity of $N_H$ itself follows immediately, using the fact that $H$ disconnects $X$.
\item We first claim that for each vertex $x$ of $N_H$ there exists a unique edge  $e_x$ dual to $H$ and containing $x$.  The existence is by definition of $N_H$, and we just have to check uniqueness. Assume by contradiction that there are two distinct edges $e,e'$ containing $x$ and dual to $H$. Then the endpoints $y,y'$ of $e,e'$ distinct from $x$ are contained in the same connected boundary component of $N_H$. By  Theorem 4.13 of \cite{Sageev95} this boundary component is convex. Thus $(y,x,y')$ is not a geodesic, so that $d(y,y')\le 1$. The complex $X$ is simply-connected and its 2-faces are polygons with even length: it follows that the length mod. 2 of paths in $X$ depends only on the endpoints. Thus $d(y,y')$ is even, and we deduce that $y=y'$, so that $e=e'$, contradiction.

Let $Q$ denote any cube of $X$ dual to $H$. Then by the previous remark any two edges of $Q$ dual to $H$ in $X$  are in fact parallel inside $Q$. Let $\sigma_Q$ denote the reflection of $Q$ preserving each edge of $Q$ dual to $H$, and exchanging the endpoints of these edges. For $Q_1\subset Q_2$  the restriction of $\sigma_{Q_2}$ to $Q_1$ is $\sigma_{Q_1}$. Thus the collection of reflections $({\sigma_Q})_{Q{\rm\ dual\ to\ }H}$ defines a reflection $\sigma_H:N_H\to N_H$ with the desired properties.

\end{enumerate}
\end{proof}

\begin{defn}
Let $\gamma=(x_0,x_1,\dots,x_n)$ denote a non stuttering  path.
We first let $(e_1,\dots,e_n)$ denote the sequence of edges of $X$ such that the vertices of $e_i$ are $v_{i-1},v_i$. A hyperplane {\em crosses} $\gamma$ iff it is dual to one of the edges $e_i$, and {\em the sequence of hyperplanes that $\gamma$ crosses} is  the sequence $(H_1,\dots,H_n)$ where $H_i$ is the hyperplane of $X$ dual to $e_i$.
\end{defn}

The last part of Theorem 4.13 in \cite{Sageev95} gives:

\begin{thm}\label{thm:combgeod}
Let $X$ denote  a cubing.  Then a non stuttering path is a combinatorial geodesic if and only if the sequence of hyperplanes it crosses has no repetition.

In particular the combinatorial distance between two vertices $x,y$ is equal to the number of hyperplanes of $X$ that separates $x$ and $y$.
\end{thm}

\section{Combinatorial translation length.}\label{sec: comb trans length}

Let $f$ denote an automorphism of the $CAT(0)$ cube complex $X$. Then $f$ is an isometry for the $CAT(0)$ distance, but also for the combinatorial distance.

\begin{defn}\label{defn:translength}
Let $f\in{\rm Aut}(X)$. For every point $x\in X$ we denote by $\delta^0(f,x)$ the $CAT(0)$ distance between $x$ and $f(x)$. And for every vertex $p\in X^0$ we denote by $\delta(f,p)$ the combinatorial distance between $p$ and $f(p)$. We then set $\delta^0(f)=\inf_{x\in X}\delta^0(f,x)$ and $\delta(f)=\inf_{p\in X^0}\delta(f,p)$. We call $\delta^0(f)$ the {\em $CAT(0)$ translation length of $f$},
and $\delta(f)$ the {\em combinatorial translation length of $f$}. Clearly translation lengths are a conjugation invariant.

Note that in fact the  combinatorial translation length $\delta(f)$ is defined for an arbitrary automorphism $f$ of a graph
(which is not necessarily the 1-skeleton of a cubing).

\end{defn}

Here is a straightforward proof that Baumslag-Solitar groups $BS(m,n)$ with $m\neq n$ cannot act properly on a finite dimensional $CAT(0)$ cube complexe $X$.

Consider  the $CAT(0)$ translation length of the element $b$ acting on $X$. In $CAT(0)$ cube complexes, finite dimensionality implies that the set of shapes of $X$ is finite: thus  every isometry $f$ satisfies  $\delta^0(f^n)=n\delta^0(f)$ (see \cite{BridsonHaefliger}). So we see that the relation $ab^ma^{-1}=b^n$ implies $m\delta^0(b)=n\delta^0(b)$. Thus $\delta^0(b)=0$ and $b$ has a fixed point since there are no parabolic elements (see \cite{BridsonHaefliger}). Hence the full infinite cyclic subgroup $<b>$ fixes a point in $X$.

I first thought that  isometries of arbitrary $CAT(0)$ cube complexes should have a nice behaviour too. But then  I was told the following example  by Estelle Souche (it also appears in an exercise of \cite{BridsonHaefliger}):

\begin{exmp}

 Consider the Hilbert space $\ell^2(\integers)$. Let
$(e^k)_{k\in\integers}$ denote the Hilbert basis such that ${e^k}_n=1$ if
$n=k$ and ${e^k}_n=0$ otherwise. The set $V=\integers^{(\integers)}$ of
maps $\integers \to\integers$ with finite support is a subset of $H$. We
consider the unit cubes of $\ell^2(\integers)$ whose vertices are elements of $V$, and whose edges are parallel to one of the ${e^k}$'s: the
union of all these cubes is a $CAT(0)$ cube complex $X$. Note that $X$ is
not finite dimensional.

 Define an isometry $\sigma$ of $\ell^2(\integers)$ on the Hilbert basis by
$\sigma(e^k)=e^{k+1}$. The map $f:\ell^2(\integers)\to \ell^2(\integers)$ defined by $f(u)=e_0+\sigma(u)$
is an affine isometry of $\ell^2(\integers)$. Note that $f(V)=V$ and $\sigma$ preserves $(e^k)_{k\in\integers}$, thus $f$ induces an
isometry of $X$.

 The isometry $f$ has no fixed point in $\ell^2(\integers)$. 

 For each $k\in \naturals$ define a vector $u^k\in X$ as follows:
${u^k}_n=0$ if $n<0$ or $n>k$ and  ${u^k}_n=1-{n\over k}$ if $0\le n\le
k$. Then $\delta^0(f,u^k)=\sqrt{1\over k}$, thus $\delta^0(f)=0$.

 In the terminology of Bridson-Haefliger (\cite{BridsonHaefliger}) the isometry $f\in{\rm
Aut}(X)$ is parabolic.

 Note that $\delta(f)>0$ since $f$ has no fixed points, and $f(0)=e_0$,
so in fact $\delta(f)=\delta(f,0)=1$. If we set $p_n=f^n(0)$ then
$(p_n)_{n\in\integers}$ is a combinatorial infinite geodesic preserved by $f$, on
which $f$ has unit combinatorial translation length.

\end{exmp}

\begin{defn}[elliptic, hyperbolic]

 Let $f\in {\rm Aut}(X)$. We say that $f$ is {\em combinatorially
elliptic} if $f$ has a fixed point in $X^0$. We say that $f$ is {\em
combinatorially hyperbolic} if $f$ is not elliptic and $f$  preserves
some infinite combinatorial geodesic $\gamma$ on which it acts as a non trivial translation. Any such geodesic $\gamma$ will be called an {\em axis for $f$}.

\end{defn}

 \begin{exmp}\label{exmp:inversion}

 Let $X$ denote a single edge. Then the  automorphism of $X$ exchanging
the endpoints of the edge is neither elliptic nor hyperbolic.

\end{exmp}

\section{Actions without inversion.}
\label{sec:inversion}

 In order to get rid of the trouble caused by automorphisms similar to
the one described in Example~\ref{exmp:inversion} we first introduce the
corresponding notion:

 \begin{defn}[inversions]
 Let $f$ denote an automorphism of a $CAT(0)$ cube complex $X$. Let $H$
denote a hyperplane of $X$, and let $X^+,X^-$ denote the two strict
half-spaces defined by $H$. We say that {\em $f$ has an inversion along $H$}
whenever $f(X^+)=X^-$ (and thus $f(X^-)=X^+,f(H)=H$). We say that {\em $f$ acts
without inversion} if there is no hyperplane  $H$ such that $f$  has an inversion
along $H$.

We say that the automorphism $f$ acts {\em stably without
inversion} when $f$ and each power of $f$ act without inversion.

And we say that {\em a group $G$ of automorphisms of $X$  acts without inversion}
if all of its elements act without inversion. Note that if a group acts without inversion, then any of its elements act stably without inversion.

\end{defn}

Just as in the tree case we have:
\begin{lem}\label{lem:subdinversion}

 Let $f$ denote an automorphism of some $CAT(0)$ cube complex $X$. Then
$f$ acts without inversion on the cubical subdivision $X'$.

\end{lem}
\begin{proof}
Every  edge $e$ of $X'$ joins the center of a cube $Q(e)$ of $X$ to the center of one of its codimension 1 face.
 For each edge $e$ of $X'$,  denote by ${X'}^+(e)$ the strict half-space of $X'$ containing the center of $Q(e)$ but not the center of its codimension 1 face.

For any automorphism $f$ of $X$ and any edge $e$ of $X'$ we have $f({X'}^+(e))={X'}^+(f(e))$.

The Lemma follows because if $e_1,e_2$ are opposite edges of a square of $X'$ then ${X'}^+(e_1)={X'}^+(e_2)$.
\end{proof}

\begin{question}

Is there a finitely  generated group $G$ acting on a locally compact $CAT(0)$ cube complex all of whose finite index subgroups have an inversion ?
\end{question}

\section{Automorphisms preserving a geodesic.}
\label{sec:axisautom}

\begin{prop}\label{prop:ifaxis}

Let $\mathcal G$ denote any graph, and let $\gamma$ denote an infinite combinatorial
geodesic of $\mathcal G$. If an automorphism $f$ of $\mathcal G$ preserves $\gamma$ then
\begin{enumerate}
\item  either $f$ has a fixed point in $\gamma$
\item or $f$ exchanges  two consecutive vertices of $\gamma$
\item or  there is a number $d\in\integers,d\neq 0$   such that
$${\rm for\ every\ } n\in\integers, {\rm we\ have\ } f(p_n)=p_{n+d}$$
and furthermore in that case for every $n\in\integers$ we have $\delta(f)=\delta(f,p_n)=\size{d}$.
\end{enumerate}

\end{prop}

 Note that the last property shows that for any other $f$-invariant
geodesic $\gamma'$, the translation length of $f$ on $\gamma'$ is
$\delta(f)$ too. Note also that when $\mathcal G$ is the 1-skeleton of a cubing, the second possibility in the Lemma above corresponds to an inversion. We thus get:

\begin{cor}\label{cor:ifaxis} Let $X$ denote a cubing.
 Assume that $f\in{\rm Aut}(X)$ acts without inversion
and is combinatorially hyperbolic. Then $f$ has the same translation length $d$ on each axis, and in fact $d=\delta(f)$.
Furthermore for any integer $n> 0$ the automorphim $f^n$ is hyperbolic, each axis for $f$ being an axis for $f^n$, and we have $\delta(f^n)=n\delta(f)$.
\end{cor}

\begin{proof}[Proof of Lemma~\ref{lem:ifaxis}]

  There is a bijection $\phi:\integers\to\integers$ such that
$f(p_n)=p_{\phi(n)}$. Since $f$ is an automorphism we have
$\size{\phi(n+1)-\phi(n)}=1$. Thus $\phi(n)=d+\varepsilon n$, with
$\varepsilon\in\{-1,1\}$.

 Assume first $\varepsilon=-1$. Either $d$ is even: then $f(p_{d\over
2})=p_{d\over 2}$. Or $d$ is odd: then $f$ echanges the adjacent vertices
$p_{d-1\over 2},p_{d+1\over 2}$.

 Assume now $\varepsilon=1$. If $d=0$ then $f$ fixes each point of
$\gamma$. Else for every $n\in\integers$ we  have $f(p_n)=p_{n+d}$. Let us
prove in this case that  for every $n\in\integers$ we have
$\delta(f)=\delta(f,p_n)=\size{d}$.

 Up to replacing $({p}_n)_{n\in\integers}$ by
$({p}_{-n})_{n\in\integers}$, we may and will assume that $d>0$.

 So $f$ acts on $\gamma$ as a translation of length $d$. Let $x$ denote any
vertex of $\mathcal G$. Then
$d(p_0,f^n(p_0))=nd$ and by the triangle inequality $nd(x,f(x))\ge d(x,f^n(x))\ge
d(p_0,f^n(p_0)) - d(p_0,x)-d(f^n(p_0),f^n(x))= nd - 2d(p_0,x)$. If we devide by $n$ and let
$n$ tend to infinity we obtain the desired inequality $d(x,f(x))\ge d$.

\end{proof}

The previous argument on arbitrary graphs  is due to Thomas Elsner. Initially we had proven a more precise result only valid for automorphisms of $CAT(0)$ cube complexes:

\begin{lem}\label{lem:ifaxis}
Let $X$ denote a cubing and let $f$ be an automorphism of $X$. Assume that $f$ preserves an infinite geodesic $(p_n)_{n\in\integers}$ and $f$ acts by translation $p_n\mapsto p_{n+d}$ on $\gamma$ (with $d\in\naturals^*$).
Then for every vertex $p$ there exists an infinite geodesic $(q_n)_{n\in\integers}$ preserved by $f$, on which $f$ acts by $q_n\mapsto q_{n+d}$, and such that for all $n\in\naturals$ we have $d(p,q_n)=d(q,p_0)+n$. In particular $d(p,f(p))\ge d$, so that $\delta(f)=d$.
\end{lem}
\begin{proof}

Consider the set $\Gamma_d$ of all
combinatorial geodesics $\gamma'=({p'}_n)_{n\in\integers}$ such that
$f(\gamma')=\gamma'$, and $f$ acts on $\gamma'$ by
$f({p'}_n)={p'}_{n+d}$. Pick a vertex $p$ of $X$. By assumption
$\Gamma_d$ is not empty, so there is a geodesic $\gamma\in\Gamma_d$
minimizing the distance $d(p,\gamma')$. To simplify notations we denote
such a geodesic of $\Gamma_d$  by $\gamma=({q}_n)_{n\in\integers}$.

 Let $q_n$ a vertex of $\gamma$ such that $d(p,q_n)\le d(p,q_m)$ for all
integers $m\in\integers$. Up to shifting indices we may assume that $n=0$.  Consider a geodesic $\gamma'$ joining $p$ to
$q_n$. Assume that there exists an integer $n\le m\le n+d$ such that the
path $\gamma'(q_n,\dots,q_m)$ is not geodesic. Observe that necessarily
$n<m$. Choose a minimal such integer $m$.

By Theorem~\ref{thm:combgeod}   there is a hyperplane $H$ crossing $\gamma'(q_n,\dots,q_m)$  at least twice. By minimality of $m$ the path
  $\gamma'(q_n,\dots,q_{m-1})$ is geodesic. Hence again by Theorem~\ref{thm:combgeod} the  hyperplane  $H$ separates
  $\{q_{m-1},q_m\}$ and crosses
$\gamma'(q_n,\dots,q_{m-1})$  
   exactly once. In fact $H$ cannot cross $(q_n,\dots,q_{m-1})$ since
  $(q_n,\dots,q_{m})$ is geodesic.

 Then the maximal subpath $\gamma_H$ of $\gamma'(q_n,\dots,q_m)$
containing $q_n$ and not crossed by $H$ is contained in the geodesic
$\gamma'(q_n,\dots,q_{m-1})$ so  $\gamma_H$  has to be geodesic. 

 Observe that by maximality the endpoints of $\gamma_H$ are in $N_H$;
the terminal point is $q_{m-1}$. By Theorem~\ref{thm:hyperplanes}(\ref{convneighbourhood}), $N_H$ is  convex, thus we have in fact
$\gamma_H\subset N_H$. In particular $(p_n,\dots,p_{m-1})\subset N_H$.
Let $({q'}_n,\dots,{q'}_{m-1})\subset N_H$ denote  the symmetric of
$(q_n,\dots,q_{m-1})$ with respect to $H$ (note that
${q'}_{m-1}=q_{m+1}$). Then $q_n$ and ${q'}_n$ are adjacent by an edge
dual to $H$. The  hyperplane  $H$ crosses twice the  path $\gamma'(q_n,{q'}_n)$.
Thus (by Theorem~\ref{thm:combgeod})  we have $d(q,{q'}_n)<d(q,q_n)$.

 The product path
$(q_n,{q'}_n)({q'}_n,\dots,{q'}_{m-1})(q_{m},\dots,q_{n+d})$ is still a
geodesic joining $q_n$ to $q_{n+d}=f(q_n)$. We denote by $({q''}_i)_{n\le
i\le n+d}$ the vertices of this path. For any integer  $k\in\integers$ we
define ${q''}_k=f^q({q''}_r)$ with $q,r$ uniquely defined by $n\le r<n+d$
and $k=r+qd$. Then $\gamma'=({q''}_n)_{n\in\integers}$  is a well-defined
infinite path because $f(q_n)=q_{n+d}$. Note that if $k\equiv n\ [d]$
then $q_k={q''}_k$. So $\gamma'$ is an infinite geodesic because it has a
geodesic interpolation.

 By construction $f(\gamma')=\gamma'$ and $f({q''}_n)={q''}_{n+d}$. Thus
$\gamma'\in\Gamma_d$. This contradicts the minimality of $d(p,\gamma)$.

 It follows from the previous argument that $\gamma'(q_n,\dots,q_{n+d})$
is  geodesic. Thus $d(p,q_{n+d})=d(p,q_n)+d$. By triangular inequality
$d(p,f(p))\ge d(p,q_{n+d})-d(q_{n+d},f(p))=(d(p,q_n)+d)- d(f(q_n),f(p))
=d$. Thus $\delta(f,p)\ge d$ and clearly $d=\delta(f,p)$ for any vertex
of a geodesic $\gamma\in\Gamma_d$.

\end{proof}

\section{Classification of automorphisms acting stably without inversion.}
\label{sec:classification}

Our main technical result is the following:
\begin{lem}\label{lem:axis}

 Let $f$  denote an automorphism of some $CAT(0)$ cube complex $X$. Let
$p$ denote a vertex of $X$ such that $\delta(f)=\delta(f,p)$. If $f$
and each power of $f$  act
without inversion then for any integer $n\ge 0$ we have  $d(p,f^n(p))=n\delta(f)$.

\end{lem}

 The condition ``without inversion'' is necessary in view of
Example~\ref{exmp:inversion}. The condition ``stably without inversion'' is also necessary:
consider an order four rotation of a square. 
\begin{proof}
We may assume $\delta(f)>0$. We then argue  by contradiction. So assume that there is a vertex $p$ and a positive integer $n$ such that $\delta(f)=d(p,f(p))$ and $d(p,f^n(p))\neq n\delta(f)$. By triangular inequality we always have $d(p,f^n(p))\le n\delta(f)$. Thus 
$d(p,f^n(p))\neq n\delta(f)$ means $d(p,f^n(p))<n\delta(f)$. Consider a pair $(p,n)$ with the smallest possible  integer $n$. Observe that $n\ge 2$.

 Fix a combinatorial geodesic $\gamma_0=(p_0=p,p_1,\dots,p_\delta=f(p))$.
Set $\gamma_i=f^i(\gamma_0)$. The path
$\gamma=\gamma_0\gamma_1\dots\gamma_{n-1}$ joins $p$ to $f^n(p)$, and is
not  a geodesic. Thus by Theorem~\ref{thm:combgeod} there exists a  hyperplane  $H$ of $X$ crossing at least
twice the path $\gamma$.

 By minimality of $n$ the subpaths $\gamma_0\gamma_1\dots\gamma_{n-2}$
and $\gamma_1\gamma_2\dots\gamma_{n-1}$ are geodesic. Thus by Theorem~\ref{thm:combgeod} the hyperplane $H$ crosses these
subpaths of $\gamma$ only once. This implies that $H$ crosses $\gamma_0$
once, $H$  crosses $\gamma_{n-1}$ once and $H$ does not cross at all the
subpath $\gamma_1\gamma_2\dots\gamma_{n-2}$.

The path $\gamma$ contains three maximal subpaths not crossed by $H$. Two of these subpaths contain the endpoints of $\gamma$ and we let $\gamma_H$ denote the third one (in the middle). We now choose the  hyperplane  $H$ crossing twice $\gamma$ such that the
length of $\gamma_H$ is minimal. Then we claim that the corresponding
subpath $\gamma_H$ is a geodesic: for else there would be a  hyperplane  $H'$
crossing twice $\gamma_H$, hence also $\gamma$, and we would have
$\gamma_{H'}\subset\gamma_{H},\gamma_{H'}\neq\gamma_{H}$, contradicting
the minimality of the length of $\gamma_H$.

 Note that by maximality of the subpath $\gamma_H$ the endpoints of
$\gamma_H$ are in $N_H$ (not separated by $H$). By Theorem~\ref{thm:hyperplanes}(\ref{convneighbourhood}) the neighbourhood  $N_H$ is convex, thus
we have $\gamma_H\subset N_H$. Let $\sigma_H$ denote the path of $N_H$
symmetric to $\gamma_H$ with respect to $H$. The initial vertex of
$\sigma_H$ is one of the vertices $p_i$ of $\gamma_0$. Let $q'$ denote
the vertex of $\sigma_H$ adjacent to $f(p)$; the edge joining $q'$ to
$f(p)$ is dual to $H$. We let $q$ denote the vertex adjacent to $p$ such
that $f(q)=q'$, and we denote by $K$ the  hyperplane  dual to the edge $a$ joining
$q$ and $p$.

By symmetry inside $N_H$ (see Theorem~\ref{thm:hyperplanes}(\ref{symmetry})) we see that the vertex
$f(q)$ is on a geodesic from $p$ to $f(p)$. Let  $\gamma'$ denote the part of this geodesic from $p$ to $f(q)$: it has length $\delta(f)-1$. Consider now the path $\gamma''=(q,p)\gamma'$. The length of this path is $\delta(f)$, and it joins $q$ to $f(q)$. Thus in fact $d(q,f(q))=\delta(f)$ and $(q,p)\gamma'$ is a geodesic. In particular the  hyperplane  $K$ separates $\{ q,f(q)\}$.
 
Consider now the product path $\gamma''f(\gamma'')\cdots f^{n-2}(\gamma'')$ joining $q$ to $f^{n-1}(q)$, and of length $(n-1)\delta(f)$. Since $d(q,f(q))=\delta(f)$, by  minimality of $n$ we see that $\gamma''f(\gamma'')\cdots f^{n-2}(\gamma'')$ has to be a geodesic. In particular  the  hyperplane  $K$ separates $\{ q,f^{n-1}(q)\}$.
 
We claim that in fact $K$ separates $\{ q,f^{n-1}(p)\}$. Otherwise $K$  separates $f^{n-1}(p)$ and $f^{n-1}(q)$. Thus $K$ is dual to the edge $f^{n-1}(a)$. Since $K$ is also dual to $a$ we   get $f^{n-1}(K)=K$. But since $K$ separates $\{q,f^{n-1}(q)\}$ we see that $f^{n-1}$ has an inversion along $K$, contradiction.

Since $K$ separates $\{ q,f^{n-1}(p)\}$, when we apply $f$ we see that $H$ separates $\{ f(q),f^{n}(p)\}$. Thus $H$ does not separate $\{ f(p),f^{n}(p)\}$. This is a contradiction, because, as we already noticed, the path $\gamma_1\gamma_2\dots\gamma_{n-1}$ is a geodesic, and it is crossed exactly once by $H$.

\end{proof}

\begin{cor}\label{cor:delta}
 Assume that $f\in{\rm Aut}(X)$ acts stably without inversion and has no fixed point. Then  $f$ is combinatorially hyperbolic. More precisely $f$  has an axis through each vertex $p$ minimizing $d(p,f(p))$. For any
integer $n\ge 0$, each axis for $f$ is an axis for $f^n$ and $\delta(f^n)=n\delta(f)$.
\end{cor}

\begin{proof}

 Let $p$ denote any of the vertices of $X$ such that $d(p,f(p))=\delta(f)$. For brevity we write $\delta(f)=d$.  
 
Let $\gamma_0=(x_0,x_1,\dots,x_d)$ denote any combinatorial geodesic from $p$ to $f(p)$ (so that in particular $x_d=f(x_0)$).
For any integer  $k\in\integers$ we
define ${p}_k=f^q({x}_r)$ with $q,r$ uniquely defined by $0\le r<d$
and $k=r+qd$. Note that for $k=0,1,\dots,d$ we have $p_k=x_k$, and for an arbitrary $k$ we have $f(p_k)=p_{k+d}$. Thus $\gamma=({p_k})_{k\in\integers}$  is an 
infinite path.  The map $x\mapsto d(x,f(x))$ achieves its minimal value at $p=p_0$, and thus at each vertex  $p_{kd},k\in\integers$. By Lemma~\ref{lem:axis} it follows that  the finite subpath $({p_k})_{k_1d\le k\le k_2d}$  is a geodesic (for any pair $(k_1,k_2)\in\integers^2$ with $k_1\le k_2$). Thus $\gamma$ is an infinite geodesic, and
by construction $\gamma$ is invariant under $f$.

We conclude by applying Corollary~\ref{cor:ifaxis}.

\end{proof}

We have proved:

 \begin{thm}\label{thm:noparabolic} Every automorphism of a $CAT(0)$ cube complex acting stably without
inversion is either  combinatorially elliptic or  combinatorially
hyperbolic. 

\end{thm}

\begin{rem}

 Let $f$ denote an automorphism of a $CAT(0)$ cube complex $X$. Assume
that $f$ has an inversion along a  hyperplane  $H$ of $X$. Then either $f$ is
elliptic on $X'$ and the set of its fixed point is contained in the
subcomplex  $H\subset X'$, or $f$ is hyperbolic on
$X'$ and all the axes of $f$ are inside $H$.

\end{rem}

\section{Applications.}
\label{sec:application}

We now prove Theorem~\ref{thm:distortednopropercube} of the Introduction.

\begin{proof}
So let $a\in\Gamma$ denote an infinite order element such that $\frac{\size{a^n}}{n}\to 0$. Assume $\Gamma$ acts on a $CAT(0)$ cube complex $X$. We claim that $a$ has a fixed point in $X$, so that the action of $\Gamma$ is not proper.

By Corollary~\ref{cor:delta} we have $\delta(a^n)=n\delta(a)$.

Now for any decomposition $g=s_1\dots s_k$ we clearly have $\delta(g)\le \sum_{i=1}^{i=k}\delta(s_i)$. Consider a geodesic decomposition ${a^n}=s_1\dots s_{k_n}$ on the finite set  $S$ of generators of $\Gamma$. We deduce $\delta(a^n)\le {k_n}\max_{s\in S}\delta(s)$. Since $\lim\frac{k_n}{n}=0$ and $\delta(a^n)=n\delta(a)$ we have $\delta(a)=0$, which concludes the proof.

\end{proof}

\def\polhk#1{\setbox0=\hbox{#1}{\ooalign{\hidewidth
  \lower1.5ex\hbox{`}\hidewidth\crcr\unhbox0}}} \def\cprime{$'$}
  \def\cprime{$'$} \def\polhk#1{\setbox0=\hbox{#1}{\ooalign{\hidewidth
  \lower1.5ex\hbox{`}\hidewidth\crcr\unhbox0}}}

\end{document}